\documentclass{amsart}

\title{The Brouwer fixed point theorem and the Borsuk--Ulam theorem}
\author{Anthony Carbery}
\thanks{}
\date{Revised May 2010}
\usepackage{amssymb}
\usepackage{amsmath}
\usepackage{url}

\newtheorem{theorem}{Theorem}
\newtheorem{proposition}{Proposition}
\newtheorem{lemma}{Lemma}
\newtheorem{corollary}{Corollary}

\begin{document}

\maketitle

\section{Brouwer fixed point theorem}

\medskip
\noindent
The Brouwer fixed point theorem states that every continuous map 
$f: \mathbb{D}^n \rightarrow \mathbb{D}^n$ has a fixed point. When 
$n=1$ this is a trivial consequence of the intermediate value theorem.

\medskip
\noindent
In higher dimensions, if not, 
then for some $f$ and all $x \in \mathbb{D}^n$, $f(x) \neq x$. So the 
map $\tilde{f}: \mathbb{D}^n \rightarrow \mathbb{S}^{n-1}$ obtained by 
sending $x$ to the unique point on $\mathbb{S}^{n-1}$ on the line segment 
starting at $f(x)$ and passing through $x$ is continuous, and when restricted 
to the boundary $\partial \mathbb{D}^n = \mathbb{S}^{n-1}$ is the identity.

\medskip
\noindent
So to prove the Brouwer fixed point theorem it suffices to show there is no 
map $g : \mathbb{D}^n \rightarrow \mathbb{S}^{n-1}$ which restricted to the 
boundary $\mathbb{S}^{n-1}$ is the identity. (In fact, this is an equivalent 
formulation.)

\medskip
\noindent
It is enough, by a standard approximation argument, to prove this for $C^1$ 
maps $g$. Consider

$$\int_{\mathbb{D}^n} \det Dg = \int_{\mathbb{D}^n} d g_1 \wedge 
d g_2 \wedge \dots \wedge d g_n$$
where $Dg$ is the derivative matrix of $g$. On the one hand this is zero
as $Dg$ has less than full rank at each $x \in \mathbb{D}^n$, and on the other 
hand it equals, by Stokes' theorem,

$$\int_{\mathbb{S}^{n-1}} g_1 d g_2 \wedge \dots \wedge d g_n.$$
This quantity manifestly does not depend on the behaviour of $g_1$ in
the interior of $\mathbb{D}^n$, and, by symmetry, likewise depends
only on the restrictions of $g_2, \dots , g_n$ to $\mathbb{S}^{n-1}$.
But on $\mathbb{S}^{n-1}$, $g$ is the identity $I$, so that reversing 
the argument, this quantity also equals
$$\int_{\mathbb{D}^n} \det D I = |\mathbb{D}^n|.$$

\medskip
\noindent
In do Carmo's book {\em{Differential Forms and Applications}}, an 
argument along the above lines is attributed essentially to E. Lima. 
Is there a similarly simple proof of the Borsuk--Ulam theorem via
Stokes' theorem?

\section{Borsuk--Ulam theorem}

\medskip
\noindent
The Borsuk--Ulam theorem states that for every continuous map 
$f: \mathbb{S}^n \rightarrow \mathbb{R}^n$ there is some $x$ with 
$f(x) = f(-x)$. Once again, when $n=1$ this is a trivial consequence 
of the intermediate value theorem.

\medskip
\noindent
In higher dimensions, we first note that it suffices to prove this for
smooth $f$. Knowing it for smooth functions, uniformly approximate continuous 
$f: \mathbb{S}^n \rightarrow \mathbb{R}^n$ by smooth maps $f_m$, for each
of which there is an $x_m$ with $f_m(x_m) = f_m(-x_m)$; then there is a
subsequence (which we'll also call $x_m$) convergent to some 
$x \in \mathbb{S}^n$.
Then $f_m(x_m) = (f_m(x_m) - f(x_m)) + f(x_m)$; the first term is as
small as we please for $m$ sufficiently large, and the second term
converges to $f(x)$ by continuity of $f$. So $f_m(x_m) \rightarrow
f(x)$ and similarly  $f_m(-x_m)
\rightarrow f(-x)$.

\medskip
\noindent
So assume for contradiction that $f$ is smooth and that for all $x$, 
$f(x) \neq f(-x)$. Then 
$$\tilde{f}(x) := \frac{f(x) - f(-x)}{|f(x) - f(-x)|}$$ 
is a smooth map $\tilde{f} : \mathbb{S}^{n} \rightarrow
\mathbb{S}^{n-1}$ such that 
$\tilde{f}(-x) = - \tilde{f}(x)$ for all $x$, i.e. 
$\tilde{f}$ is {\em{odd, antipodal}} 
or {\em{equivariant}} with respect to the map $x \mapsto -x$.

\medskip
\noindent
Thus it suffices to show there is no equivariant smooth map 
$h : \mathbb{S}^n \rightarrow \mathbb{S}^{n-1}$. Equivalently, 
thinking of the closed upper hemisphere of $\mathbb{S}^n$ as 
$\mathbb{D}^n$, it suffices to show there is no smooth map 
$g: \mathbb{D}^n \rightarrow \mathbb{S}^{n-1}$ which is equivariant 
on the boundary. Once again this formulation is equivalent to the 
Borsuk--Ulam theorem and shows (since the identity map is equivariant) 
that it generalises the Brouwer fixed point thorem. 

\subsection{Case $n=2$}

\medskip
\noindent
As above, it's enough to show that there does not exist a $g: \mathbb{D}^2
\rightarrow \mathbb{S}^1$ which is equivariant on the boundary, i.e. such that $g(-x) =
-g(x)$ for $x \in \mathbb{S}^1$.

\medskip
\noindent
If there did exist such a (smooth) $g$, consider

$$\int_{\mathbb{D}^2} \det D g = \int_{\mathbb{D}^2} d g_1 \wedge d g_2$$
On the one hand this is zero as $Dg$ has less than full rank at each $x$, 
and on the other hand it equals, by Stokes' theorem,

$$ \int_{\mathbb{S}^1} g_1 d g_2 = - \int_{\mathbb{S}^1} g_2 d g_1$$

\medskip
\noindent
So it's enough to show that

$$\int_0^1 (g_1(t) g_2^{\prime}(t) - g_2(t) g_1^{\prime}(t))  dt \neq 0$$
for $g = (g_1,g_2) : \mathbb{R}/\mathbb{Z} \rightarrow \mathbb{S}^1$ 
satisfying $g(t + 1/2) = - g(t)$ for all $0 \leq t \leq 1$.

\medskip
\noindent
Now clearly 
$$(g_1(t) g_2^{\prime}(t) - g_2(t) g_1^{\prime}(t)) dt $$
represents the element of net arclength for the curve $(g_1(t),g_2(t))$
measured in the anticlockwise direction. Indeed, $|g| = 1$ implies
$\langle g, g^\prime \rangle = \frac{1}{2} \frac{d}{dt} |g|^2 = 0$, so that $\det(g,
g^{\prime}) = \pm |g||g^{\prime}| = \pm |g^{\prime}|$, with the plus
sign occuring when $g$ is moving anticlockwise. 
And the point of the equivariance condition is
that $(g_1(1/2),g_2(1/2)) = - (g_1(0),g_2(0))$, and that 
$$\int_0^1 g_1(t) g_2^{\prime}(t) dt = 2 \int_0^{1/2} g_1(t) g_2^{\prime}(t) dt.$$
In passing from $(g_1(0),g_2(0))$ to $(g_1(1/2),g_2(1/2))$ the total net
arclength traversed is clearly an odd multiple of $\pi$, and so we're done.

\medskip
\noindent
Note that the argument shows that for any smooth equivariant $h : \mathbb{S}^1 
\rightarrow \mathbb{S}^1$ we have
$$\int_{\mathbb{S}^1} h_1 d h_2 = - \int_{\mathbb{S}^1} h_2 d h_1 \neq
0,$$
and indeed is an odd multiple of $\pi$.

\medskip
\noindent
In the higher dimensional case $n \geq 3$ we'd be done by the same argument
if we could show that 

$$ \int_{\mathbb{S}^{n-1}} g_1 d g_2 \wedge d g_3 \dots \wedge d g_n \neq 0$$
whenever $g: \mathbb{S}^{n-1} \rightarrow \mathbb{S}^{n-1}$ is equivariant.

\subsection{Dimensional reduction}\label{dimred}
The material in this subsection is based upon notes by Shchepin
available at
\url{http://www.mi.ras.ru/~scepin/elem-proof-reduct.pdf} 
with some details added, and is adapted to the smooth case under consideration.

\begin{theorem}
Suppose $n \geq 4$ and there exists a smooth equivariant map $f: \mathbb{S}^n \rightarrow 
\mathbb{S}^{n-1}$. Then there exists a smooth equivariant map $\tilde{f}: \mathbb{S}^{n-1} 
\rightarrow \mathbb{S}^{n-2}$.
\end{theorem}

\medskip
\noindent
Once this is proved, only the case $n=3$ of the Borsuk--Ulam theorem remains 
outstanding. 
Note that only the case when $f$
is surjective is interesting here (as if $f$ avoids a pair of points
we can simply retract its image into an equator). 

\begin{proof}
Starting with $f$, we shall identify suitable equators $E_{n-1} \subseteq 
\mathbb{S}^{n}$ and $E_{n-2} \subseteq \mathbb{S}^{n-1}$, and build a
smooth equivariant map $\tilde{f} : E_{n-1} \rightarrow E_{n-2}$.

\medskip
\noindent
We first need to know that there is some pair of antipodal points $\{\pm A \}$
in the target $\mathbb{S}^{n-1}$ whose preimages under $f$ are at most
``one-dimensional''. 
This is intuitively clear by dimension counting but for rigour we
appeal to Sard's theorem.\footnote{Shchepin works instead with polyhedra homeomorphic to
spheres and piecewiese linear surjections between them. 
We may assume that the image of each face in the domain is contained in a face of 
the target. In this setting linear algebra shows that there is a pair of antipodes 
whose inverse images are finite unions of line segments.} Indeed,
Sard's theorem tells us that the image under $f$ of the set $\{x \in
\mathbb{S}^{n} :  \mbox{ rank }Df(x) < n-1 \}$ is of Lebesgue
measure zero: so there are plenty of points $A \in \mathbb{S}^{n-1}$
at all of whose preimages $x$ -- if there are any at all -- $Df(x)$ 
has full rank $n-1$. By the implicit function theorem, for each such 
$x$ there is a neighbourhood $B(x,r)$ such that $B(x,r) \cap
f^{-1}(A)$ is diffeomorphic to the interval $(-1,1)$. The whole 
of the compact set $f^{-1}(A)$ is covered by such balls, from 
which we can extract a finite subcover: so indeed $f^{-1}(A)$
is covered by finitely many diffeomorphic copies of $(-1,1)$.

\medskip
\noindent
Now for $x \in f^{-1}(A)$ and $y \in f^{-1}(-A)= - f^{-1}(A)$ with $y \neq - x$, 
consider the unique
geodesic great circle joining $x$ to $y$. The family of such is clearly indexed 
by the two-parameter family of points of $f^{-1}(A)\times f^{-1}(-A) \setminus\{(x, -x) \; : \;
f(x) = A \}$. 
Their union is therefore a manifold in $\mathbb{S}^n$ of dimension at most three. Since 
$n \geq 4$ there must be points $\pm B \in \mathbb{S}^n$
{\em{outside}} this union
(and necessarily outside $f^{-1}(A) \cup f^{-1}(-A)$). 
Such a point has the property that {\em{no}} geodesic great circle passing through 
it meets points of both $f^{-1}(A)$ and $f^{-1}(-A)$ other than possibly at antipodes. 
In particular, no {\em{meridian}} joining $\pm B$ meets both $f^{-1}(A)$ and 
$f^{-1}(-A)$.

\medskip
\noindent
We now identify $E_{n-2}$ as the equator of $\mathbb{S}^{n-1}$ whose
equatorial plane is perpendicular to the axis joining $A$ to $-A$; and
we identify $E_{n-1}$ as the equator of $\mathbb{S}^{n}$ whose
equatorial plane is perpendicular to the axis joining $B$ to $-B$.
We assume for notational simplicity that $B$ is the north pole 
$(0,0, \dots, 0, 1)$.

\begin{lemma}\label{one} 
Suppose $B = (0,0, \dots, 0, 1) 
\in \mathbb{S}^n$ and that  $X \subseteq \mathbb{S}^n$ is a closed 
subset such that no meridian joining $\pm B$ meets both $X$ and $-X$. 
Let $\mathbb{S}^{n}_{\pm}$ denote the open upper and lower hemispheres 
respectively. 
Then there is an equivariant diffeomorphism 
$\psi: \mathbb{S}^n \rightarrow \mathbb{S}^n$ 
such that 
$$ X \subseteq \psi (\mathbb{S}^{n}_{+}).$$ 
\end{lemma}

\begin{proof}
Since $X$ is closed and $\pm B \notin X$, there is an $\epsilon >0$ such that 
spherical caps centred 
at $\pm B$ subtending angles of $2 \epsilon $ at the origin do not meet $\pm X$.
Consider the projections of $\pm X$ along meridians from $B$ and $-B$ on the 
equator $E$ of $\mathbb{S}^n$ which lies in the plane perpendicular 
to the axis joining $\pm B$; call these $\pm \Pi(X)$. 
These are disjoint closed sets and are therefore positively separated. For $ x \in E$ let 
$d_{\pm}(x)$ be the geodesic distance in $E$ to $\pm \Pi(X)$, and let
$$ \theta(x) = (\pi/2 - \epsilon)\frac{d_+(x) - d_-(x)}{d_+(x) + d_-(x)}.$$ 
Then for all $x \in E$, $\theta(-x) = - \theta(x)$ and $|\theta(x)| \leq \pi/2 - \epsilon$; 
for $x \in \Pi(X)$ we have $\theta(x) = - \pi/2 + \epsilon$ and for  $x \in -\Pi(X)$ we have 
$\theta(x) = \pi/2 - \epsilon$. Mollify $\theta$ if necessary to
obtain a {\em{smooth}} function $\tilde{\theta}: E \rightarrow [- \pi/2 +
  \epsilon /2 , \pi/2 - \epsilon /2]$ satisfying $\tilde{\theta}(-x) =
-\tilde{\theta}(x)$ and such that for  $x \in \Pi(X)$ we have
$- \pi/2 + \epsilon/2 \leq  \tilde{\theta}(x) \leq  - \pi/2 + 3 \epsilon/2 $. 

\medskip
\noindent
For $x \in E$ let $\psi(x)$ be the point of 
$\mathbb{S}^n$ on the meridian through $x$ with latitude 
$\tilde{\theta}(x)$. Then $\psi : E \rightarrow \mathbb{S}^n$ is a smooth 
equivariant map and clearly we can extend this to be an equivariant 
diffeomorphism $\psi$ of $\mathbb{S}^n$ to itself which maps
$\mathbb{S}^{n-1}_+$ to the region above $\psi(E)$. 

\medskip
\noindent
Finally, since for $x \in X $ we have $\tilde{\theta}(x) \leq - \pi/2 +
3 \epsilon /2$, we have that $X$ is contained in the region above
$\psi(E)$, that is, $X \subseteq \psi(\mathbb{S}^{n}_+)$.

\end{proof}

\medskip
\noindent
Continuing with the proof of the theorem, we apply the lemma with 
$X = f^{-1}(A)$. Let $\phi$ be restriction of $\psi$ to $E =
E_{n-1}$. Consider the restriction $\widehat{f}$ of $f$ to $\phi(E)$: 
it has the property that $\widehat{f}(\phi(E))$ does not contain $\pm
A$. Let $r$ be the standard retraction of 
$\mathbb{S}^{n-1} \setminus \{\pm A \}$ onto its equator $E_{n-2}$; finally 
let $\tilde{f} = r \circ \widehat{f} \circ \phi$, which is clearly
smooth and equivariant.

\end{proof}

\medskip
\noindent
{\bf{Remark.}} It is clear from the construction of $\psi$ in Lemma
\ref{one} that we may assume that it fixes meridians and acts as the
identity on small neighbourhoods of $\pm B$. Moreover it is also clear
that we may find a smoothly varying family of diffeomorphisms $\psi_t$
of $\mathbb{S}^n$, $0 \leq t \leq 1$, such that $\psi_0$ is the
identity and $\psi_1 = \psi$.

\subsection{Case $n=3$}

The treatment in Proposition \ref{prop} and Lemma \ref{two}
below is also based upon the notes of Shchepin peviously cited,
but it incorporates some significant differences of detail, and is adapted 
to the smooth case under consideration.

\begin{proposition}\label{prop}
Suppose that $f : \mathbb{S}^3 \rightarrow \mathbb{S}^2$ is a smooth
equivariant map. Then there exists a smooth 
${f}^\dagger : \mathbb{D}^3 \rightarrow \mathbb{S}^2$ which is equivariant on 
$ \partial \mathbb{D}^3 = \mathbb{S}^2$, and moreover maps $\mathbb{S}^2_{\pm}$ to itself.
\end{proposition}

\begin{proof}
Firstly, by identifying the closed upper hemisphere of $\mathbb{S}^3$ with the closed
disc $\mathbb{D}^3$ we obtain a smooth map
$$ \widehat{f}: \mathbb{D}^3 \rightarrow \mathbb{S}^2$$
which is equivariant on $\partial \mathbb{D}^3 = \mathbb{S}^2.$

\medskip
\noindent
Then the restriction of $\widehat{f}$ to $\partial \mathbb{D}^3 = \mathbb{S}^2$
gives a smooth equivariant map $g : \mathbb{S}^2 \rightarrow \mathbb{S}^2$.
If we could take $g$ to be the identity, we would be finished (by the
argument given for the Brouwer fixed point theorem). Further, if $g$
is such that {\em{either}} for all $x$, $g(x) \neq x$, {\em{or}} for
all $x$, $g(x) \neq -x$, we can use the map
$$ \frac{ \theta g(x) \pm (1- \theta) {x}}{|\theta g(x) \pm (1-
  \theta) {x}|}$$ 
to ``graft the identity onto the outside of $\widehat{f}$'' to obtain a map
to which the Brouwer argument applies. So the ``bad'' $g$ are those
for which there exist $x$ and $y$ with $g(x) = x$ and $g(y) = -y$. 
For such $g$ we might hope instead to be able to graft an equivariant 
hemisphere-preserving $g^{\dagger}$ onto the outside of $\widehat{f}$
-- or what is almost as good, to find such a $g^{\dagger}$ and a
smooth family of diffeomorphisms $\psi_t: \mathbb{S}^2 \rightarrow
\mathbb{S}^2$ such that $\psi_0$ is the identity and $\psi_1 = \psi$
where $g^\dagger(x) \neq - g \circ \psi(x)$ for all $x$.

\medskip
\noindent
\begin{lemma}\label{two} 
If $g : \mathbb{S}^2 \rightarrow \mathbb{S}^2$ is a smooth equivariant map, then there exists 
a smooth equivariant ${g}^\dagger : \mathbb{S}^2 \rightarrow \mathbb{S}^2$ which preserves the upper
and lower hemispheres of $\mathbb{S}^2$, and a smooth family of
diffeomeorphisms $\psi_t$ such that $\psi_0 = I$, $\psi_1 = \psi$ 
and such that for all $x$, 
$$ g^\dagger(x) \neq - g \circ \psi(x).$$
\end{lemma}


\begin{proof}
Pick points $\pm A$ in the target $\mathbb{S}^2$ whose inverse images
under $g$ are at most finite. This we can do by Sard's theorem as the image
under $g$ of the set $\{ x \in \mathbb{S}^2 : \mbox{ rank }
Dg(x) < 2 \}$ has Lebesgue measure zero in $\mathbb{S}^2.$ So there are plenty
of points $A \in \mathbb{S}^2$ at all of whose preimages $x$ (if there are
any at all), $Dg(x)$ has full rank $2$. By the inverse function theorem, for each such $x$ 
there is a neighbourhood $B(x,r)$ such that $B(x,r) \cap g^{-1}(A) = \{x\}.$ 
The whole of the compact set $g^{-1}(A)$ is covered by such balls,
from which we can extract a finite subcover: so indeed $g^{-1}(A)$
consists of (at most) finitely many points. We may assume without loss
of generality that $A$ is the north pole.

\medskip
\noindent
Pick now $\pm B$ in the domain $\mathbb{S}^2$ such
that the merdinial projections of the members of $g^{-1}(\pm A)$ on
the equator whose plane is perpendicular to the axis joining $\pm B$ are
distinct. Assume without loss of generality that $B$ is the north
pole. Now applying Lemma \ref{one} (and the remarks following it) with $X = g^{-1}(A)$ we see
there is an equivariant diffeomorphism $\psi : \mathbb{S}^2 \rightarrow
\mathbb{S}^2$ 
such that $g^{-1}(A) \subseteq \psi(\mathbb{S}^{2}_+)$.
Thus $g_\ast := g \circ \psi$ is a smooth equivariant map from $\mathbb{S}^2$ to
itself and $g_\ast^{-1}(A) \subseteq \mathbb{S}^{2}_+$. Moreover $\psi
= \psi_1$ where $\psi_t$ is a smooth family of diffeomorphisms such
that $\psi_0 = I$.



\medskip
\noindent
Let $E$ be the equator and define $\tilde{g} : E \rightarrow E$ as the 
meridinial projection of $g_\ast(x)$ on $E$. It is well-defined
since $g_\ast^{-1}(\pm A) \cap E = \emptyset $. Then $\tilde{g}$ is
clearly smooth and equivariant.

\medskip
\noindent
We next extend $\tilde{g}$ to a small strip around $E$. For
$x \in \mathbb{S}^2$ let $l(x) \in [-\pi/2, \pi/2]$ denote its 
latitude with respect to $E$, and (if $x \neq \pm B$) let 
$\overline{x}$ denote its meridinial projection on $E$. 
For $0 < r < \pi/2$ let $E_r = \{x \in \mathbb{S}^2  :  |l(x)| \leq
r\}.$ Consider only $r$ so small that $E_r \cap g_\ast^{-1}(\pm A) 
= \emptyset$. Let $d = \mbox{dist }(\pm A, g_\ast(E))$.
Since $g_\ast$ is uniformly continuous, there is an $r>0$ such that for $x \in E_r$ 
we have $d(g_\ast (x), g_\ast(\overline{x})) < d/10 $. 
Now extend $\tilde{g}$ to $E_r$ by defining 
$\tilde{g}(x)$ to be the point with the same longitude (merdinial projection) 
as $\tilde{g}(\overline{x})$ and with latitude $ \pi l(x) /2r$. This extension is still
equivariant and smooth.

\medskip
\noindent
Finally extend $\tilde{g}$ to a map from $\mathbb{S}^2$ to itself by defining 
$\tilde{g}(x) = A$ for $l(x)> r$ and $\tilde{g}(x) = - A$ for $l(x)< -r$.
We see that $\tilde{g}$ is continuous, equivariant, preserves the upper 
and lower hemispheres and -- except possibly on the sets $\{ x : l(x)
= \pm r \}$ -- is smooth.

\medskip
\noindent
Consider, for $x \in E_r$, the three points $g_\ast (x)$, 
$g_\ast(\overline{x})$ and $\tilde{g}(x)$. Now $g_\ast (\overline{x})$
and $\tilde{g}(x)$ lie on the same meridian, and $g_\ast
(\overline{x})$ is distant at least $d$
from $\pm A$. On the other hand, ${g}_\ast(x)$ is at most $d/10$ 
from ${g}_\ast(\overline{x})$. Thus $g_\ast(x)$ is at least $9d/10$
from $\pm A$ and lives in a $d/10$-neighbourhood of the common meridian
containing $g_\ast (\overline{x})$ and $\tilde{g}(x)$. So for $x \in
E_r$, $g_\ast(x)$ cannot equal $-\tilde{g}(x)$. 
For $l(x) > r$ we have $\tilde{g}(x) = A$ and 
$g_\ast(x) \neq -A$ because $g_\ast^{-1}(-A)$ is contained in the lower hemisphere. Similarly
for $l(x) < -r$, $\tilde{g}(x) \neq -g_\ast(x)$. Thus for all $x \in \mathbb{S}^2$ 
we have $\tilde{g}(x) \neq -g_\ast (x)$. By continuity, $\tilde{g}(x)$ and
$-g_\ast (x)$ are positively separated on $\mathbb{S}^2$. So for any sufficiently small
uniform perturbation $\tilde{\tilde{g}}$ of $\tilde{g}$ we will have the
same property $\tilde{\tilde{g}}(x) \neq - g_\ast (x)$ for all $x$.
 
\medskip
\noindent
We now mollify $\tilde{g}$ in small neighbourhoods of $\{ x : l(x) =
\pm r \}$ (and then renormalise to ensure that the target space
remains $\mathbb{S}^2$ !) to obtain $g^{\dagger}$
which is smooth, equivariant, preserves the upper and lower hemispheres
and, being uniformly very close to $\tilde{g}$, is such that 
$g^{\dagger}(x) \neq -g_\ast(x)$ for all $x$.




\end{proof}


\medskip
\noindent
Finishing now with the proof of the proposition, for $x \in \mathbb{S}^2$ 
and $0 \leq t \leq 1$, 
let $\tilde{f}: \mathbb{D}^3 \rightarrow \mathbb{S}^2$ be defined by 

$$\tilde{f}(tx) = \frac{(3t-2){g}^\dagger(x) + (3 - 3t)(g \circ \psi) (x)}
{|(3t-2){g}^\dagger(x) + (3 - 3t)(g \circ \psi) (x)|} \mbox{ when } 2/3 \leq t \leq 1,$$

$$\tilde{f}(tx) = g \circ \psi_{3t-1}(x) \mbox{ when } 1/3 \leq t \leq 2/3 $$
\noindent
and
$$ \tilde{f}(tx) = \widehat{f}(3t x) \mbox{ when } 0 \leq t \leq 1/3.$$




\medskip
\noindent
This makes sense because for all $2/3 \leq t \leq 1$ we have 
$(3t-2){g}^\dagger(x) + (3 - 3t)(g \circ \psi) (x) \neq  0$; then $\tilde{f}$ has
all the desired properties of $f^\dagger$ (including continuity)
except possibly for smoothness at $t
= 1/3$ and $2/3 $. To rectify this we mollify $\tilde{f}$ in a small
neighbourhood of $\{ t = 1/3 \}$ and $\{ t = 2/3 \}$ and renormalise 
once more to ensure that the target space is indeed still
$\mathbb{S}^2$. The resulting $f^\dagger$ now has all the properties we need.
\end{proof}

\medskip
\noindent
Let $C$ be the cylinder $\mathbb{D}^2 \times [-1,1]$ in $\mathbb{R}^3$
with top and bottom faces ${D}_{\pm}$ and curved vertical
boundary $V = \mathbb{S}^1 \times [-1,1]$. Let $S_\pm$ be the upper and
lower halves of $S = \partial C$. Let $E$ be the equator of $S$.


\medskip
\noindent
Now $C$, with the all points on each vertical line of $V$ identified, 
is diffeomorphic to $\mathbb{D}^3$, and $S$ is also diffeomorphic
to $\mathbb{S}^2$, so that we immediately get:

\begin{corollary}
Under the same hypotheses as Proposition \ref{prop}, there exists a
smooth map from $C$ to $S$ which is equivariant on $\partial C$, 
which is constant on vertical lines in $V$ and which maps 
${D}_{\pm}$ into $ S_{\pm}.$
\end{corollary}

\medskip
\noindent
The final argument needed to complete the proof of the Borsuk--Ulam
theorem is:

\begin{proposition}
There is no smooth map $f: C \rightarrow S$ which is equivariant
on $\partial C$, which is constant on vertical lines in $V$ and which maps 
${D}_{\pm}$ into $ S_{\pm}$.
\end{proposition}

\begin{proof}
If such an $f$ existed, then

$$\int_C \det Df = \int_{C} d f_1 \wedge 
d f_2 \wedge d f_3$$
where $Df$ is the derivative matrix of $f$. On the one hand this is zero
as $Df$ has less than full rank at almost every $x \in C$, and on the other 
hand it equals, by Stokes' theorem,

$$\int_{\partial C} f_3 \; d f_1 \wedge d f_2 = 
\int_V f_3 \;d f_1 \wedge d f_2 + 2\int_{D_+} f_3 \;d f_1 \wedge d f_2 $$
by equivariance. 

\medskip
\noindent
Now $f$ maps $V$ into $E$, so that $f_3 = 0$ on $V$, and the first term on the right vanishes.

\medskip
\noindent
As for the second term, 
$$\int_{D_+} f_3 \;d f_1 \wedge d f_2 
= \int_{D_+ \cap \{ x \; : \; f_3(x) = 1 \} } f_3 \;d f_1 \wedge d f_2 + 
\int_{D_+ \cap \{x \;: \; f_3 (x) < 1\} } f_3 \;d f_1 \wedge d f_2.$$
The region of $D_+$ on which $f_3(x) < 1$ consists of patches on
which $f_1^2(x) + f_2^2(x) = 1$, and so $2 f_1 \; d f_1 + 2 f_2 \; d f_2 =
0$. Taking exterior products with $df _1$ and $df_2$ tells us that on
such patches we have $f_1 \; d f_1 \wedge d f_2 = f_2 \; d f_1 \wedge d f_2
= 0$.  Multiplying by $f_1$ and $f_2$ and adding we get that $h \; d f_1 
\wedge d f_2 = 0$ for all $h$ supported on a patch on which $f_3(x) < 1$. 
So for any $h$ we have 
$$ \int_{D_+ \cap \{x \; : \; f_3 (x) < 1\} } h \;d f_1 \wedge d f_2 = 0.$$
Hence
$$\int_{D_+} f_3 \;d f_1 \wedge d f_2  
= \int_{D_+ \cap \{ x \; : \; f_3(x) = 1 \} } \;d f_1 \wedge d f_2$$
$$= \int_{D_+ \cap \{ x \; : \; f_3(x) = 1 \} } \;d f_1 \wedge d f_2 +
  \int_{D_+ \cap \{x \; : \; f_3(x) <1 \}} \;d f_1 \wedge d f_2  $$
$$= \int_{D_+} \;d f_1 \wedge d f_2 .$$

\medskip
\noindent
By Stokes' theorem once again we have
$$ \int_{D_+} \;d f_1 \wedge d f_2 = 
\int_{\partial D_+} f_1 d f_2 = - \int_{\partial D_+} f_2 d f_1,$$
and, since $f$ restricted to  ${\partial D_+}$ is equivariant, 
this quantity is nonzero (and indeed is an odd multiple of $ \pi$),
by the remarks in the proof of the case $n=2$ above.

\medskip
\noindent
So no such $f$ exists and we are done.

\end{proof}

\bigskip
\medskip
\noindent
{\em{Acknowledgement.}} The author would like to thank Pieter Blue, Marina
Iliopoulou and Mark Powell for helpful suggestions and comments.

\end{document}